\documentclass[a4paper,10pt, leqno]{amsart}
\usepackage{amssymb}
\usepackage{amsmath}
\usepackage{amsthm}
\usepackage{color}
\usepackage{graphicx}
\newtheorem{thm}{Theorem}[section]
\newtheorem{lem}[thm]{Lemma}

\newtheorem{rmk}{Remark}

\numberwithin{equation}{section}

\newcommand{\la}{\lambda}
\newcommand{\va}{\varphi}

\newcommand{\wa}{\widetilde{a}}
\newcommand{\wb}{\widetilde{b}}
\newcommand{\wc}{\widetilde{c}}

\newcommand{\R}{\mathbb{R}} 
\newcommand{\N}{\mathbb{N}} 
\newcommand{\pp}{\partial}

\allowdisplaybreaks

\title[]{%
An Inverse problem 
and 
a time-like Carleman estimate 
for 
parabolic integro-differential equations. 
}
\author{Atsushi Kawamoto}
\address{ Department of Mathematical Sciences, The University
of Tokyo, Komaba Meguro Tokyo 153-8914, Japan
}
\email{kawamo@ms.u-tokyo.ac.jp
}
\date{}
\begin{document}

\maketitle

\begin{abstract}
In this article, 
We investigate an inverse problem 
of determining the time-dependent source factor in parabolic integro-differential equations from boundary data. 
We establish the uniqueness and the conditional stability estimate of H\"older type for the inverse source problem in a cylindrical shaped domain. 
Our methodology is based on the Bukhgeim-Klibanov method by means of the Carleman estimate. 
Here we also derive the ``time-like'' Carleman estimate for parabolic integro-differential equations.  
\end{abstract}

%

%
%
%
%
\section{Introduction}


Let 
$T>0$, $n \in \N$,  $\Omega\subset \R^n$ be a bounded domain with sufficiently smooth boundary $\pp\Omega$, 
We set $Q= (0,T) \times \Omega$ and $\Sigma= (0,T) \times \pp\Omega$. 
We denote  by $\nu=\nu(x)$. the outwards unit normal vector to $\pp\Omega$ at $x\in \pp\Omega$ 

We consider 
the following parabolic integro-differential equations:
\begin{equation}
\label{eq:intpara_eq}
\pp_t u(t,x) - L u(t,x) - \int_0^t K(t,\tau)u(\tau,x)\,d\tau 
=g(t,x), \quad
(t,x)\in Q, 
\end{equation}
where $L$  is the uniformly elliptic operator:
\begin{equation*}
L
=
\sum_{i,j=1}^n \pp_i (a_{ij}(t,x)\pp_j ) 
+\sum_{j=1}^n b_j (t,x) \pp_j 
+ c(t,x),\quad (t,x)\in Q,
\end{equation*}
and $K(t,\tau)$ is the following memory kernel operator: 
\begin{align*}
&K(t,\tau)
=
\sum_{i,j=1}^n \pp_i (\wa_{ij}(t,\tau,x)\pp_j ) 
+\sum_{j=1}^n \wb_j (t,\tau,x) \pp_j 
+ \wc(t,\tau,x),
\quad
(t,\tau,x)\in (0,T)^2\times \Omega,
\end{align*}
with $a_{ij}\in C^1 (\overline{Q})$, $b_j, c \in L^\infty (Q)$,
$\wa_{ij}\in C^1 (\overline{(0,T)^2\times \Omega})$, $\wb_j, \wc \in L^\infty ((0,T)^2\times \Omega)$ 
for $i,j=1,\ldots,n$. We assume that 
$a_{ij}=a_{ji}$ 
for $i,j=1,\ldots,n$ 
and that 
there exists a constant $\mu>0$ such that 
\begin{align*}
&\frac1\mu |\xi|^2 \leq \sum_{i,j=1}^n a_{ij}(t,x) \xi_i \xi_j \leq \mu |\xi|^2,			& (t,x) \in Q, \quad \xi  \in \R^n. 
\end{align*}

In this article,  we use the parabolic integro-differential operator $P$ defined by 
\begin{equation*}
P u(t,x)
=
\pp_t u(t,x) - L u(t,x) - \int_0^t K(t,\tau)u(\tau,x) \,d\tau . 
\end{equation*}
We define the co-normal derivative for $L$ by 
$
\frac{\pp u}{\pp \nu_L} = \sum_{i,j=1}^n a_{ij} (\pp_i u)\nu_j. 
$


The parabolic integro-differential equations \eqref{eq:intpara_eq} 
arise from various studies such as 
the anomalous diffusion/transport in the heterogeneous media \cite{DT2006, FP2016}, 
the thermal conduction in materials with memory \cite{N1971, P1993} 
and so on. 
As an example, 
if we investigate the non-Fickian diffusion with memory effects, 
we may derive integro-differential equations of parabolic type \eqref{eq:intpara_eq} 
as the hybrid model for the classical and anomalous diffusion. 
(see \cite{FP2016}). 
On the well-posedness in the direct problems 
for parabolic integro-differential equations under appropriate assumptions for $L$, $K$, $g$ and initial and boundary value conditions, 
we may find a lot of previous results  
(see e.g., \cite{AT1986, AH2014, KJ2011, LP1992, Z1993} and the references therein). 


In this article, 
we firstly derive a time-like Carleman estimate for \eqref{eq:intpara_eq}. 
A Carleman estimate is a weighted $L^2$ inequality and 
a priori estimate for partial differential equations. 
A Carleman estimate has various applications such as inverse problems, unique continuations, the control theory (see e.g., \cite{H, Is1, KT, LRS}). 
Here the ``time-like'' means that the weight function of the Carleman estimate is depending on only the time variable $t$ 
and independent of the space variables $x$. 
In \cite{K1978, LRS, LP1961}, 
we may find this kind of the time-like estimates  
which are used to prove the unique continuation for partial differential equations. 
As a remark, we may refer to the weight energy method in \cite{AS, P}  
in which we may see the idea similar to time-like Carleman estimates. 
Related to the applications to inverse problems, 
Yamamoto \cite{Y} introduced the time-like Carleman estimate 
with the $x$-independent weight function $e^{-\la t}$ where $\la$ is a large parameter and 
he applied it to an inverse source problem and a backward problem for parabolic equations. 
As for the time-like Carleman estimate, we also refer to \cite{Kawa} in which   
the estimate for the degenerate parabolic equations was derived. 

In relation to the Carleman estimates for the parabolic integro-differential equations, 
we refer to  works by A. Lorenzi et al. \cite{L2012, L2013a, L2013b, LL2014,  LLY2016, LM2012}.  
In their article, they treat the case of 
that the integrands of the Volterra type integral in the equations are the first or zero-th order in space, 
that is, the case of $\wa_{ij}=0$ for $i,j=1,\ldots,n $  in  \eqref{eq:intpara_eq}. 


Our strategy to prove the time-like Carleman estimate is mainly relies on the result for the parabolic case. 
In our case, however, the difficulty arises from an integral term  
containing second order partial derivative in space. 
To estimate the integral term in \eqref{eq:intpara_eq}, 
we introduce a key integral inequality containing the weight function of the Carelman estimate 
(see e.g., Lemma 3.1 in \cite{K2}, Lemma 3.3.1 in \cite{KT}). 
Using the time-linke Carleman estimate for parabolic equations in \cite{Y} and the key integral inequality, 
we prove our Carleman estimate for \eqref{eq:intpara_eq}.


Secondly we discuss an inverse source problem for the parabolic integro-differential equations \eqref{eq:intpara_eq}. 
We assume that $\Omega$ is a cylindrical shaped domain and 
the coefficients of $L$ is independent of the one component of space variables. 
And we establish the uniqueness and the conditional stability estimate of H\"older type 
in the determination of the time dependent source factor from boundary data.


There are a lot of previous researches on inverse problems for parabolic integro-differential equations. 
In particular, a number of articles are intensively devoted to the identifications of memory kernels  (see e.g., \cite{C2009, CL1997, JW1998, JW1999, JW2000, LP1992, LS1988} and the references therein).
Related to inverse source problems, we may refer to \cite{JK2009, KJ2011, LM2010}. 
In \cite{KJ2011}, they consider the inverse problem of determining the source term from spatial data at the final time. 


However, there are not many works on the uniqueness and the stability results 
using boundary data as observations.  
Moreover, stability estimates in inverse source problems are not studied well for parabolic integro-differential equations. 


To prove the uniqueness and the stability in our inverse problem, 
we adopt the Bukhgeim-Klibanov method. 
Bukhgeim and Klibanov \cite{BK} established the global uniqueness results in inverse problems 
by means of the Carleman estimate. 
More precisely, we may refer to some monographs and surveys \cite{Is1, K1, K2, KT} and references therein. 
To apply the Bukhgeim-Klibanov method, 
our key tools are the time-like Carleman estimate and some integral inequalities. 
Most related to our result, we refer to \S9 in \cite{Y} in which the time-like Carelman estimate 
is applied to prove the uniqueness in the inverse source problem for parabolic equations in a cylindrical shaped domains. 
Here we may also refer to a related work \cite{GK2013}.



%
%

The plan of this article is as follows: 
In section 2, we state and show the time-like Carleman estimate for \eqref{eq:intpara_eq} (Theorem 2.3) 
by using the key integral inequality (Lemma2.1). 
In section 3, we state the inverse source problem 
and give our main results, that is, the uniqueness (Theorem 3.1) and the stability estimate (Theorem 3.2) in our inverse problem. 
In section 4, we prove our main theorems (Theorem 3.1 and 3.2).

%
%
\section{Carleman Estimate}
In this section, we state and show 
the time-like Carleman estimate for 
the parabolic integro-differential equation.

We consider the following weight function:
\begin{equation*}
\va (t)=e^{-\la (t-T)}, \quad t \in (0,T) . 
\end{equation*}

Firstly we prove the following key integral inequality. 
\begin{lem}
\label{lem:keyineq01}
For all $z \in L^2 (Q)$, we have  
\begin{align*}
&\int_Q \left( \int_0^t z(\tau,x) d\tau \right)^2 e^{2s\va}\,dtdx \\
&\leq 
C\int_Q \frac1{s\la\va}|z(t,x)|^2 e^{2s\va}\,dtdx
+
C\int_Q \frac1{s^2\la}|z(t,x)|^2 e^{2s\va}\,dtdx.
\end{align*}
\end{lem}
\begin{proof}
We note that 
\begin{equation*}
\va \geq 1, \quad 
\pp_t (e^{2s\va}) = -2 s\la\va e^{2s\va}, \quad
\pp_t \left(\frac{1}{\va}\right)=\frac{\la}{\va} \leq \la \quad \text{in $(0,T)$. }
\end{equation*}
By the Cauchy-Schwarz inequality and the integration by parts, we have
\begin{align*}
&\int_Q \left( \int_0^t z(\tau,x) d\tau \right)^2 e^{2s\va}\,dtdx\\
&\leq T \int_Q \left( \int_0^t |z(\tau,x)|^2 d\tau \right)  e^{2s\va}\, dtdx
\\
&=-\frac{T}{2s\la} \int_\Omega \int_{0}^T \frac1{\va} \left( \int_0^t |z(\tau,x)|^2 d\tau \right) \pp_t (e^{2s\va})\, dtdx\\
&=-\frac{T}{2s\la} \int_\Omega \frac1{\va(T)} \left( \int_0^T |z(\tau,x)|^2 d\tau \right) e^{2s\va(T)}\, dx \\
&\quad 
+\frac{T}{2s\la} \int_\Omega \int_{0}^T \pp_t \left(\frac1{\va}\right) \left( \int_0^t |z(\tau,x)|^2 d\tau \right) e^{2s\va}\, dtdx \\
&\quad 
+\frac{T}{2s\la} \int_\Omega\int_{0}^T \frac1{\va} |z(t,x)|^2 e^{2s\va}\,dtdx\\
&\leq
\frac{T}{2}
\int_Q\frac1{s\la \va} |z(t,x)|^2  e^{2s\va} \,dtdx
+\frac{T}{2s} \int_Q \left( \int_0^t |z(\tau,x)|^2 d\tau \right)\va e^{2s\va}\, dtdx \\
&=
\frac{T}{2}
\int_Q\frac1{s\la \va} |z(t,x)|^2  e^{2s\va} \,dtdx
-\frac{T}{4s^2\la} \int_\Omega \int_{0}^T\left( \int_0^t |z(\tau,x)|^2 d\tau \right) \pp_t (e^{2s\va})\, dtdx \\
&=
\frac{T}{2}
\int_Q\frac1{s\la \va} |z(t,x)|^2  e^{2s\va} \,dtdx
-\frac{T}{4s^2\la} \int_\Omega \left( \int_0^T |z(\tau,x)|^2 d\tau \right) e^{2s\va(T)}\, dtdx \\
&\quad
+\frac{T}{4s^2\la} \int_\Omega \int_{0}^T |z(t,x)|^2 e^{2s\va}\, dtdx \\
&\leq
C
\int_Q\frac1{s\la \va} |z(t,x)|^2  e^{2s\va} \,dtdx
+
C \int_Q \frac1{s^2\la} |z(t,x)|^2 e^{2s\va}\, dtdx . 
\end{align*}
Thus we conclude this lemma. 
\end{proof}

To derive our main Caleman estimate, 
we use the time-like Carleman estimate established by Yamamoto \cite{Y}. 
We may obtain the following inequality in the proof of Theorem 9.3 in \cite{Y}. 
\begin{lem}
\label{lem:tce01} 
There exists $\la_0>1$ such that for any $\la>\la_0$, 
we may choose $C=C(\la_0)>0$ satisfying the following inequality 
\begin{align}
\nonumber
&\int_Q \left[ \frac{1}{s\va} \left( |\pp_t u|^2 + \sum_{i,j=1}^n |\pp_i\pp_j u|^2 \right) + \la |\nabla u|^2 + s\la^2\va |u|^2 \right] e^{2s\va}\,dtdx\\
&\leq C 
\int_Q |(\pp_t -L)u|^2 e^{2s\va}\,dtdx 
+C\int_{ \Sigma} s\la \va \left(|\pp_t u| + |u| \right) \left|\frac{\pp u}{\pp \nu_L} \right|e^{2s\va}\,dtdS \nonumber \\
&\quad 
+C\int_Q  \la^2 |u|^2e^{2s\va}\,dtdx, 
\label{eq:tce01} 
\end{align}
for all $s>1$ and all $u \in H^1(0,T;H^2(\Omega))$ with $u(0,x)=u(T,x)=0$, $x\in \Omega$. 
\end{lem}
%
%

Using Lemma \ref{lem:keyineq01} and Lemma \ref{lem:tce01}, 
we may prove the following Carleman estimate for the parabolic integro-differential equation \eqref{eq:intpara_eq}. 

\begin{thm}
There exists $\la_0>1$ such that for any $\la>\la_0$, 
we may choose $s_0(\la)>1$ satisfying:
there exists $C^\prime=C^\prime(s_0,\la_0)>0$ such that 
\begin{align}
\nonumber
&\int_Q \left[ \frac{1}{s\va} \left( |\pp_t u|^2 + \sum_{i,j=1}^n |\pp_i\pp_j u|^2 \right) + \la |\nabla u|^2 + s\la^2\va |u|^2 \right] e^{2s\va}\,dtdx\\
&\leq C^\prime 
\int_Q |Pu|^2 e^{2s\va}\,dtdx 
+C^\prime\int_{\Sigma} s\la \va \left(|\pp_t u| + |u| \right) \left|\frac{\pp u}{\pp \nu_L} \right|e^{2s\va}\,dtdS,  
\label{eq:tce_intpara} 
\end{align}
for all $s>s_0$ and all $u \in H^1(0,T;H^2(\Omega))$ with $u(0,x)=u(T,x)=0$, $x\in \Omega$. 
\end{thm}
\begin{proof}
Note that 
\begin{equation*}
\pp_t u(t,x)-Lu(t,x)
=Pu(t,x) +\int_0^t K u(t,\tau)u(\tau,x)\,d\tau ,\quad (t,x)\in Q.
\end{equation*}
By Lemma \ref{lem:tce01}, we have 
\begin{align}
\nonumber
&\int_Q \left[ \frac{1}{s\va} \left( |\pp_t u|^2 + \sum_{i,j=1}^n |\pp_i\pp_j u|^2 \right) + \la |\nabla u|^2 + s\la^2\va |u|^2 \right] e^{2s\va}\,dtdx\\
&\leq C 
\int_Q |Pu|^2 e^{2s\va}\,dtdx 
+C\int_{\Sigma} s\la\va \left(|\pp_t u| + |u| \right) \left|\frac{\pp u}{\pp \nu_L} \right|e^{2s\va}\,dtdS \nonumber \\
&\quad 
+C\int_Q \left| \int_0^t  K(t,\tau)u(\tau,x)\,d\tau\right|^2 e^{2s\va}\,dtdx
+C\int_Q \la^2 |u|^2 e^{2s\va}\,dtdx. 
\label{eq:tce101}
\end{align}
We estimate the third term on the right-hand side of \eqref{eq:tce101} by Lemma \ref{lem:keyineq01}. 
\begin{align*}
&\int_Q \left|\int_0^t  K(t,\tau)u(\tau,x)\,d\tau\right|^2 e^{2s\va}\,dtdx \\
&\leq C \int_Q \left| \int_0^t  \sum_{i,j=1}^n |\pp_i\pp_j u(\tau,x)|\,d\tau\right|^2 e^{2s\va}\,dtdx \\
&\quad
	+ C \int_Q \left| \int_0^t  \left( |\nabla u(\tau,x)|+|u(\tau,x)|\right)\,d\tau\right|^2 e^{2s\va}\,dtdx\\
&\leq C\int_Q \frac{1}{s\la\va} \sum_{i,j=1}^n |\pp_i\pp_ju|^2 e^{2s\va}\,dtdx 
+ C\int_Q \frac{1}{s^2\la} \sum_{i,j=1}^n |\pp_i\pp_j u|^2 e^{2s\va}\,dtdx \\
&\quad 
+ C\int_Q \left( |\nabla u|^2 + |u|^2 \right)e^{2s\va}\,dtdx , 
\end{align*}
for $\la>\la_0$ and $s>1$. 
Together this with \eqref{eq:tce101}, we may obtain
\begin{align}
\nonumber
&\int_Q \left[ \frac{1}{s\va} \left( |\pp_t u|^2 + \sum_{i,j=1}^n |\pp_i\pp_j u|^2 \right) + \la |\nabla u|^2 + s\la^2\va |u|^2 \right] e^{2s\va}\,dtdx\\
&\leq C 
\int_Q |Pu|^2 e^{2s\va}\,dtdx 
+\int_{ \Sigma} s\la\va \left(|\pp_t u| + |u| \right) \left|\frac{\pp u}{\pp \nu_L} \right|e^{2s\va}\,dtdS\nonumber \\
&\quad 
+ C\int_Q \left( \frac{1}{s\la\va} \sum_{i,j=1}^n |\pp_i\pp_ju|^2 + |\nabla u|^2\right) e^{2s\va}\,dtdx \nonumber \\
&\quad
+ C\int_Q \left( \frac{1}{s^2\la} \sum_{i,j=1}^n |\pp_i\pp_j u|^2 + \la^2 |u|^2 \right) e^{2s\va}\,dtdx ,  
\label{eq:tce02}  
\end{align}
for $\la>\la_0$ and $s>1$. 
Taking sufficiently large $\la>\la_0$, and then, choosing 
sufficiently large $s>1$, we may absorb the third and forth terms on the right-hand side of \eqref{eq:tce02} 
into the left-hand side. 
Thus, we conclude the \eqref{eq:tce_intpara}. 
\end{proof}

%
%
\section{Inverse Source Problem}
Let $\ell>0$, $n \in \N$, $n\geq 2$, $D\subset \R^{n-1}$ be a bounded domain 
with sufficiently smooth boundary $\pp D$. 
We set a cylindrical shaped domain and its boundaries: 
$\Omega=(0,\ell)\times D$,  $Q=(0,T)\times \Omega$, 
$\Sigma_0=(0,T)\times \{0,\ell\}\times \overline{D}$, 
$\Sigma_1=(0,T)\times (0,\ell)\times \pp D$, $\Sigma =(0,T)\times \pp\Omega=\Sigma_0\cup \Sigma_1$.  
We use the notation $x=(x_1,x^\prime)\in \Omega$ where $x_1 \in (0,\ell)$ and $x^\prime=(x_2,\ldots, x_n) \in D$. 

Let $R(t,x)$, $(t,x)\in Q$ be a given function. 
We consider 
\begin{equation}
\label{eq:ibvp}
\left\{
\begin{aligned}
&P^\prime u (t,x)=f(t,x^\prime) R(t,x),	&(t,x)\in Q,			\\
&u(0,x)=0,								&x\in \Omega,		\\
&u(t,x)=\pp_1 u(t,x)=0,					&(t,x)\in \Sigma_0,
\end{aligned}
\right.
\end{equation}
where $f\in L^2((0,T)\times D)$, 
%
%
\begin{equation*}
P^\prime u(t,x)
=
\pp_t u(t,x) - L^\prime u(t,x) - \int_0^t K(t,\tau)u(\tau,x)\,d\tau , 
\quad (t,x)\in Q,
\end{equation*} 
and $L^\prime$ and $K(t,\tau)$ are the following operators:
\begin{align*}
L^\prime 
&=
\sum_{i,j=1}^n \pp_i (a^\prime_{ij}(t,x^\prime )\pp_j ) 
+\sum_{j=1}^n b_j (t,x) \pp_j 
+ c(t,x) ,
\quad (t,x)\in Q,\\
K (t,\tau)
&=
\sum_{i,j=1}^n \pp_i (\wa_{ij}(t,\tau,x)\pp_j ) 
+\sum_{j=1}^n \wb_j (t,\tau,x) \pp_j 
+ \wc(t,\tau,x), 
\quad (t,\tau,x)\in (0,T)^2\times \Omega,
\end{align*}
with
$a_{ij}\in C^1 ([0,T]\times \overline{D})$, 
$\wa_{ij}\in C^1 ([0,T]^2; C^2 ([0,\ell]; C^1(\overline{D})))$, 
$b_j, c\in L^\infty (0,T;C^1([0,\ell];L^\infty (D)))$, 
$\wb_j, \wc \in L^\infty ((0,T)^2;C^1([0,\ell];L^\infty (D)))$ 
for $i,j=1,\ldots,n$. 
We assume that 
$a^\prime _{ij}=a^\prime _{ji}$ 
for $i,j=1,\ldots,n$ 
and that 
there exists a constant $\mu>0$ such that 
\begin{align*}
&\frac1\mu |\xi|^2 \leq \sum_{i,j=1}^n a^\prime_{ij}(t,x^\prime) \xi_i \xi_j\leq \mu |\xi|^2,		& (t,x^\prime) \in (0,T)\times D,\quad \xi  \in \R^n. 
\end{align*}
Here we note that 
$f$ and 
$a^\prime_{ij}$, $i,j=1,\ldots,n$ are 
independent of $x_1$. 

We assume that 
%
%
\begin{equation}
\label{eq:Rcondi}
R\in C^1 ([0,T];C^3([0,\ell];C^2(\overline{D}))),\quad |R(t,x)|\neq 0,\quad (t,x)\in \overline{Q}. 
\end{equation}

Let $t_0 \in (0,T)$.  We investigate the following inverse problem. 

\noindent
\textbf{Inverse Source Problem}: 
Can we determine the time dependent source factor $f(t,x^\prime)$, $(t,x^\prime)\in (0,t_0)\times D$ 
from boundary data on $\Sigma_1$?

\begin{thm}[Uniqueness]
\label{thm:uni}
Let $R$ satisfy \eqref{eq:Rcondi}. 
We assume that $u, \pp_1 u \in H^1(0,T;H^2 (\Omega))$ and $u$ satisfies \eqref{eq:ibvp}. 
Moreover we suppose that there exists a constant $M>1$ such that  
$\| u\|_{L^2(Q)}+\|\pp_1 u\|_{L^2 (Q)} \leq M$ and $\|f\|_{L^2((0,T)\times D)}\leq M$.  
If $u(t,x)=0$, $(t,x)\in \Sigma_1$, then we have $f(t,x^\prime)=0$, $(t,x^\prime)\in (0,t_0)\times D$. 
\end{thm}

\begin{thm}[Stability estimate]
\label{thm:stab}
Under the assumptions of Theorem \ref{thm:uni},  
there exist constants $\kappa \in (0,1)$ and $C>0$ such that 
\begin{equation}
\label{eq:stab}
\| f\|_{L^2((0,t_0)\times D)} 
\leq C  B^\kappa, 
\end{equation}
where
\begin{align*}
B&=
\| \pp_t \pp_1 u \|_{L^2(\Sigma_1)}
+
\| \pp_t  u \|_{L^2(\Sigma_1)}
+
\| \pp_1  u \|_{L^2(\Sigma_1)}
+
\|  u \|_{L^2(\Sigma_1)}\\
&\quad
+
\left\| \frac{\pp (\pp_1u) }{\pp \nu_{L^\prime}}  \right\|_{L^2(\Sigma_1)}
+
\left\| \frac{\pp u }{\pp \nu_{L^\prime}}  \right\|_{L^2(\Sigma_1)} . 
\end{align*}
\end{thm}

\begin{rmk}
We may not derive the uniqueness result (Theorem \ref{thm:uni}) directly by the stability estimate (Theorem \ref{thm:stab}). 
In the proof of the stability estimate, however, we may obtain the uniqueness in our inverse source problem. 
Hence we prove Theorem \ref{thm:uni} and Theorem \ref{thm:stab} in the next section  simultaneously. 
\end{rmk}

\section{Proofs of Theorems}

Setting $u=Rv$ in $\overline{Q}$, 
by \eqref{eq:ibvp}, we obtain
%
%
\begin{align}
\nonumber
&\pp_t v- L^\prime v 
-2\sum_{i,j=1}^n \frac1{R} a^\prime_{ij}(\pp_i R) \pp_j v  \\
&
+\left( \frac{\pp_t R}{R} 
- \sum_{i,j=1}^n \frac1{R} \pp_i (a_{ij}^\prime \pp_j R ) 
- \sum_{j=1}^n \frac1{R} b_j \pp_j R \right) v \nonumber \\
&
-\int_0^t \frac{R(\tau, x)}{R(t,x)} K(t,\tau) v (\tau, x)\,d\tau \nonumber \\
&
-\int_0^t \frac{1}{R(t,x)} 
\Biggl [
	2\sum_{i,j=1}^n \wa_{ij}(t,\tau, x) \pp_i R(\tau, x)\pp_j v (\tau, x)\nonumber \\
	&
	\qquad 
	+\sum_{i,j=1}^n  \pp_i (\wa_{ij}(t,\tau, x)\pp_j R (\tau, x)) v(\tau, x)
	+\sum_{j=1}^n \wb_j (t,\tau, x) \pp_j R(\tau,x) v(\tau, x)  
\Biggr]\,d\tau \nonumber \\
&=f(t,x^\prime) ,	\quad (t,x)\in Q,
\label{eq:ibvp_v} 
\end{align}
and $v(0,x)=0$,	 $x\in \Omega$,
$v(t,x)=\pp_1 v(t,x)=0$,	 $(t,x)\in \Sigma_0$. 

Differentiate the parabolic integro-differential equation of \eqref{eq:ibvp_v} with respect to $x_1$ and set $w=\pp_1 v$. 
Then we get 
%
%
\begin{align}
\nonumber 
&\pp_t w- L^\prime w 
-2\sum_{i,j=1}^n \frac1{R} a^\prime_{ij}(\pp_i R) \pp_j w  \\
&
+\left( \frac{\pp_t R}{R} 
- \sum_{i,j=1}^n \frac1{R} \pp_i (a_{ij}^\prime \pp_j R ) 
- \sum_{j=1}^n \frac1{R} b_j \pp_j R \right) w \nonumber \\
&
-\int_0^t \frac{R(\tau, x)}{R(t,x)} K(t,\tau)w (\tau, x)\,d\tau \nonumber \\
&
-\int_0^t \frac{1}{R(t,x)} 
\Biggl [
	2\sum_{i,j=1}^n \wa_{ij}(t,\tau, x) \pp_i R(\tau, x)\pp_j w (\tau, x)\nonumber \\
	&
	\qquad 
	+\sum_{i,j=1}^n  \pp_i (\wa_{ij}(t,\tau, x)\pp_j R (\tau, x)) w(\tau, x)
	+\sum_{j=1}^n \wb_j (t,\tau, x) \pp_j R(\tau,x) w(\tau, x)  
\Biggr]\,d\tau \nonumber \\
&
- 
\sum_{j=1}^n (\pp_1 b_j)\pp_j v 
+2\sum_{i,j=1}^n  \pp_1 \left( \frac1{R} a^\prime_{ij} \pp_i R \right) \pp_j v \nonumber \\
&
- \left[
 \pp_1 c 
-\pp_1 
	\left(
	\frac{\pp_t R}{R} 
	- \sum_{i,j=1}^n \frac1{R} \pp_i (a_{ij}^\prime \pp_j R ) 
	- \sum_{j=1}^n \frac1{R} b_j \pp_j R
	\right)
\right] v \nonumber \\
&
-\int_0^t \pp_1 \left(\frac{R(\tau, x)}{R(t,x)}\right)  K(t,\tau) v (\tau, x)\,d\tau \nonumber \\
&
-\int_0^t  \pp_1 \left(\frac{1}{R(t,x)}\right) 
\Biggl [
	2\sum_{i,j=1}^n \wa_{ij}(t,\tau, x) \pp_i R(\tau, x)\pp_j v (\tau, x)\nonumber \\
	&
	\qquad 
	+\sum_{i,j=1}^n  \pp_i (\wa_{ij}(t,\tau, x)\pp_j R (\tau, x)) v(\tau, x)
	+\sum_{j=1}^n \wb_j (t,\tau, x) \pp_j R(\tau,x) v(\tau, x)  
\Biggr]\,d\tau \nonumber \\
&
-\int_0^t \frac{R(\tau, x)}{R(t,x)} 
\Biggl [
	\sum_{i,j=1}^n \pp_1 \wa_{ij}(t,\tau, x) \pp_i \pp_j v (\tau, x)
	+ \sum_{i,j=1}^n \pp_i \pp_1 \wa_{ij}(t,\tau, x)  \pp_j v (\tau, x) \nonumber \\
	&
	\qquad
	+\sum_{j=1}^n \pp_1 \wb_j (t,\tau, x) \pp_j v(\tau, x)
	+ \pp_1 \wc (t,\tau, x) v(\tau, x)
\Biggr]\,d\tau 
\nonumber \\
&-\int_0^t \frac{1}{R(t,x)} 
\Biggl [
	2\sum_{i,j=1}^n \pp_1 (\wa_{ij}(t,\tau, x) \pp_i R(\tau, x)) \pp_j v (\tau, x)
	+\sum_{i,j=1}^n  \pp_1\pp_i (\wa_{ij}(t,\tau, x)\pp_j R (\tau, x)) v(\tau, x) \nonumber \\
	&\qquad
	+\sum_{j=1}^n \pp_1 (\wb_j (t,\tau, x) \pp_j R(\tau,x)) v(\tau, x)  
\Biggr]\,d\tau 
=0 ,	\quad (t,x)\in Q,			 \nonumber 
\end{align}
and 
$w(0,x)=0$, $x\in \Omega$, 
$w(t,x)=0$, $(t,x)\in \Sigma_0$. 

Here we remark the representation of $v$ by $w$. By $w(t,0,x^\prime)=0$, $(t,x^\prime)\in (0,T) \times D$ and 
\begin{equation}
\label{eq:rep_v1}
\pp_1 v(t,x)= w(t,x),\quad (t,x)\in Q, 
\end{equation}
we have
\begin{equation}
\label{eq:rep_v2}
v(t,x)=\int_0^{x_1} w (t,\zeta,x^\prime)\,d\zeta, \quad (t,x)\in Q, 
\end{equation}
and then 
\begin{equation}
\label{eq:rep_v3}
\left\{
\begin{aligned}
\pp_t v(t,x)&=\int_0^{x_1} \pp_t w (t,\zeta,x^\prime)\,d\zeta, & (t,x)\in Q, \\
\pp_j v(t,x)&=\int_0^{x_1} \pp_j w (t,\zeta,x^\prime)\,d\zeta, & (t,x)\in Q, \\
\pp_1 \pp_1 v(t,x)&=\pp_1 w (t,x),& (t,x)\in Q, \\
\pp_1 \pp_j v(t,x)&=\pp_j w (t,x),  & (t,x)\in Q, \\
\pp_i \pp_j v(t,x)&= \int_0^{x_1} \pp_i\pp_j w (t,\zeta,x^\prime)\,d\zeta, & (t,x)\in Q,
\end{aligned}
\right.
\end{equation}
for $i,j=2,\ldots, n$.

Choose $t_1, t_2 \in (0,T)$ such that $0<t_0<t_1<t_2<T$. 
Set $\delta_k =e^{-\lambda (t_k-T)}$ for $k=0,1$. 
Let $\chi \in C^\infty (\R)$ be a cut off function such that 
\begin{equation*}
0 \leq \chi (t) \leq 1, \ t \in \R, \quad
\pp_t X (t) \leq 0,\ t\in \R, \quad
\chi (t)=
\left\{
\begin{aligned}
1,\quad t<t_1,\\
0,\quad t\geq t_2. 
\end{aligned}
\right. 
\end{equation*}

Setting $y=\chi w$, we obtain
%
%
\begin{align}
\nonumber
&\pp_t y- L^\prime y 
-2\sum_{i,j=1}^n \frac1{R} a^\prime_{ij}(\pp_i R) \pp_j y  \\
&
+\left( \frac{\pp_t R}{R} 
- \sum_{i,j=1}^n \frac1{R} \pp_i (a_{ij}^\prime \pp_j R ) 
- \sum_{j=1}^n \frac1{R} b_j \pp_j R \right) y \nonumber \\
&
=
\chi \int_0^t \frac{R(\tau, x)}{R(t,x)} 
K(t,\tau) w (\tau, x)
\,d\tau \nonumber \\
&
+
\chi \int_0^t \frac{1}{R(t,x)} 
\Biggl [
	2\sum_{i,j=1}^n \wa_{ij}(t,\tau, x) \pp_i R(\tau, x)\pp_j w (\tau, x)\nonumber \\
	&
	\qquad 
	+\sum_{i,j=1}^n  \pp_i (\wa_{ij}(t,\tau, x)\pp_j R (\tau, x)) w(\tau, x)
	+\sum_{j=1}^n \wb_j (t,\tau, x) \pp_j R(\tau,x) w(\tau, x)  
\Biggr]\,d\tau \nonumber \\
&
+
\sum_{j=1}^n (\pp_1 b_j) \chi\pp_j v
-2\sum_{i,j=1}^n  \pp_1 \left( \frac1{R} a^\prime_{ij} \pp_i R \right) \chi \pp_j v \nonumber \\
&
+ \left[
 \pp_1 c 
-\pp_1 
	\left(
	\frac{\pp_t R}{R} 
	- \sum_{i,j=1}^n \frac1{R} \pp_i (a_{ij}^\prime \pp_j R ) 
	- \sum_{j=1}^n \frac1{R} b_j \pp_j R
	\right)
\right] \chi v \nonumber \\
&
+\chi \int_0^t  \pp_1 \left(\frac{R(\tau, x)}{R(t,x)}\right) 
K(t,\tau) v (\tau, x)
\,d\tau\nonumber 
\\
&+\chi \int_0^t  \pp_1 \left(\frac{1}{R(t,x)}\right) 
\Biggl [
	2\sum_{i,j=1}^n \wa_{ij}(t,\tau, x) \pp_i R(\tau, x)\pp_j v (\tau, x)\nonumber \\
	&
	\qquad 
	+\sum_{i,j=1}^n  \pp_i (\wa_{ij}(t,\tau, x)\pp_j R (\tau, x)) v(\tau, x)
	+\sum_{j=1}^n \wb_j (t,\tau, x) \pp_j R(\tau,x) v(\tau, x)  
\Biggr]\,d\tau \nonumber \\
&
+\chi \int_0^t \frac{R(\tau, x)}{R(t,x)} 
\Biggl [
	\sum_{i,j=1}^n \pp_1 \wa_{ij}(t,\tau, x) \pp_i \pp_j v (\tau, x)
	+ \sum_{i,j=1}^n \pp_i \pp_1 \wa_{ij}(t,\tau, x)  \pp_j v (\tau, x) \nonumber \\
	&
	\qquad
	+\sum_{j=1}^n \pp_1 \wb_j (t,\tau, x) \pp_j v(\tau, x)
	+ \pp_1 \wc (t,\tau, x) v(\tau, x)
\Biggr]\,d\tau \nonumber \\
&
+\chi \int_0^t \frac{1}{R(t,x)} 
\Biggl [
	2\sum_{i,j=1}^n \pp_1 (\wa_{ij}(t,\tau, x) \pp_i R(\tau, x)) \pp_j v (\tau, x)\nonumber \\
	&
	\qquad 
	+\sum_{i,j=1}^n  \pp_1\pp_i (\wa_{ij}(t,\tau, x)\pp_j R (\tau, x)) v(\tau, x) \nonumber \\
	&\qquad
	+\sum_{j=1}^n \pp_1 (\wb_j (t,\tau, x) \pp_j R(\tau,x)) v(\tau, x)  
\Biggr]\,d\tau 
=(\pp_t \chi)w,	\quad (t,x)\in Q,			 
\label{eq:ibvp_y} 
\end{align}
and $y(0,x)=y(T,x)=0$, $x\in \Omega$, $y(t,x)=0$, $(t,x)\in \Sigma_0$. 

Applying the Carleman estimate for the parabolic equations (Lemma \ref{lem:tce01}) to \eqref{eq:ibvp_y} and 
noting that the boundedness for coefficients $a^\prime_{ij}, \wa_{ij}, b_j, \wb_j, c ,\wc$ and $R$, 
we have
\begin{align}
\nonumber 
&\int_Q \left[ \frac{1}{s\va} \left( |\pp_t y|^2 + \sum_{i,j=1}^n |\pp_i\pp_j y|^2 \right) + \la |\nabla y|^2 + s\la^2\va |y|^2 \right] e^{2s\va}\,dtdx\\
&\leq C 
\int_Q |(\pp_t\chi)w|^2 e^{2s\va}\,dtdx 
+C\int_Q   \la^2 |y|^2  e^{2s\va}\,dtdx 
+C \sum_{k=1}^5 I_k +CB_1,  
\label{eq:thmce1}
\end{align}
for $\la > \la_0$ and $s>1$, 
where 
\begin{align*}
I_1&=\int_Q \left| \chi \int_0^t \sum_{i,j=1}^n |\pp_i \pp_j w (\tau, x)| \,d\tau \right|^2 e^{2s\va}\,dtdx, \\
I_2&=\int_Q \left| \chi \int_0^t \left(\sum_{j=1}^n |\pp_j w (\tau, x)| + |w(\tau, x)| \right) \,d\tau\right|^2 e^{2s\va}\,dtdx,\\
I_3&=\int_Q \left( \sum_{j=1}^n |\chi \pp_j v|^2 + |\chi v|^2 \right)e^{2s\va}\,dtdx, \\
I_4&=\int_Q \left| \chi \int_0^t \sum_{i,j=1}^n |\pp_i \pp_j v (\tau, x)| \,d\tau \right|^2 e^{2s\va}\,dtdx, \\
I_5&=\int_Q \left| \chi \int_0^t \left(\sum_{j=1}^n |\pp_j v (\tau, x)| + |v(\tau, x)| \right) \,d\tau\right|^2 e^{2s\va}\,dtdx, 
\end{align*}
and 
\begin{equation}
\nonumber 
B_1=
\int_{ \Sigma_1} s\la\va \left(|\pp_t y| + |y| \right) \left|\frac{\pp y}{\pp \nu_{L^\prime}} \right|e^{2s\va}\,dtdS. 
\end{equation}
Henceforth we estimate from $I_1$ to $I_5$ on the right-hand side of \eqref{eq:thmce1}
by using the technique in the proof of Lemma \ref{lem:keyineq01}. 

\begin{align}
\nonumber
I_1&=
\int_Q \left| \chi \int_0^t \sum_{i,j=1}^n |\pp_i \pp_j w (\tau, x)| \,d\tau \right|^2 e^{2s\va}\,dtdx \\
&
\leq 
C \int_Q \chi^2 \left(  \int_0^t \sum_{i,j=1}^n |\pp_i \pp_j w (\tau, x)|^2  \,d\tau \right) e^{2s\va}\,dtdx \nonumber \\
&=-\frac{C}{2s\la} \int_Q \chi^2 \times \frac1{\va} \left(  \int_0^t \sum_{i,j=1}^n |\pp_i \pp_j w (\tau, x)|^2  \,d\tau \right) \pp_t( e^{2s\va})\,dtdx  \nonumber \\
&=
-\frac{C}{2s\la} \int_\Omega (\chi(T))^2 \times \frac1{\va(T)} \left(  \int_0^T \sum_{i,j=1}^n |\pp_i \pp_j w (\tau, x)|^2  \,d\tau \right)  e^{2s\va(T)}\,dx \nonumber \\
&\quad
+
\frac{C}{s\la} \int_Q \chi (\pp_t \chi) \times \frac1{\va} \left(  \int_0^t \sum_{i,j=1}^n |\pp_i \pp_j w (\tau, x)|^2  \,d\tau \right)  e^{2s\va}\,dtdx  \nonumber \\
&\quad
+\frac{C}{2s\la} \int_Q \chi^2  \pp_t \left(\frac1{\va}\right) \left(  \int_0^t \sum_{i,j=1}^n |\pp_i \pp_j w (\tau, x)|^2  \,d\tau \right) e^{2s\va}\,dtdx \nonumber  \\
&\quad
+\frac{C}{2s\la} \int_Q \chi^2 \times \frac1{\va} \sum_{i,j=1}^n |\pp_i \pp_j w (t, x)|^2   e^{2s\va}\,dtdx  \nonumber \\
&\leq 
C \int_Q  \frac1{s\la\va} \sum_{i,j=1}^n |\pp_i \pp_j y (t, x)|^2   e^{2s\va}\,dtdx  \nonumber \\
&\quad 
+C \int_Q \frac1{s} \chi^2   \left(  \int_0^t \sum_{i,j=1}^n |\pp_i \pp_j w (\tau, x)|^2  \,d\tau \right)  \va e^{2s\va}\,dtdx \nonumber \\
&=
C \int_Q  \frac1{s\la\va} \sum_{i,j=1}^n |\pp_i \pp_j y (t, x)|^2   e^{2s\va}\,dtdx  \nonumber \\
&\quad 
-C \int_Q \frac1{2s^2\la} \chi^2   \left(  \int_0^t \sum_{i,j=1}^n |\pp_i \pp_j w (\tau, x)|^2  \,d\tau \right)  \pp_t (e^{2s\va})\,dtdx \nonumber \\
&=
C \int_Q  \frac1{s\la\va} \sum_{i,j=1}^n |\pp_i \pp_j y (t, x)|^2   e^{2s\va}\,dtdx  \nonumber \\
&\quad 
-C \int_\Omega \frac1{2s^2\la} \chi(T)^2   \left(  \int_0^T \sum_{i,j=1}^n |\pp_i \pp_j w (\tau, x)|^2  \,d\tau \right)  e^{2s\va(T)}\,dx \nonumber  \\
&\quad
+C \int_Q \frac1{s^2\la} \chi (\pp_t \chi)   \left(  \int_0^t \sum_{i,j=1}^n |\pp_i \pp_j w (\tau, x)|^2  \,d\tau \right)  e^{2s\va}\,dtdx \nonumber \\
&\quad
+C \int_Q \frac1{2s^2\la} \chi^2  \sum_{i,j=1}^n |\pp_i \pp_j w (t, x)|^2  e^{2s\va}\,dtdx \nonumber \\
&\leq 
C \int_Q  \frac1{s\la\va} \sum_{i,j=1}^n |\pp_i \pp_j y |^2   e^{2s\va}\,dtdx  
+C\int_Q \frac1{s^2\la} \sum_{i,j=1}^n |\pp_i \pp_j y |^2  e^{2s\va}\,dtdx,  
\label{eq:T1}
\end{align}
for $\la>\la_0$ and $s>1$.  
Here we note that $\pp_t \chi \leq 0$ in $\R$. 
\begin{align}
\nonumber 
I_2&=
\int_Q \left| \chi \int_0^t \left(\sum_{j=1}^n |\pp_j w (\tau, x)| + |w(\tau, x)| \right) \,d\tau\right|^2 e^{2s\va}\,dtdx \\
&
\leq 
C \int_Q \chi^2 \left[  \int_0^t \left(\sum_{j=1}^n |\pp_j w (\tau, x)|^2 + |w(\tau, x)|^2 \right) \,d\tau \right] e^{2s\va}\,dtdx \nonumber \\
&\leq
C \int_Q \chi^2  \left[  \int_0^t \left(\sum_{j=1}^n |\pp_j w (\tau, x)|^2 + |w(\tau, x)|^2 \right) \,d\tau \right] \va e^{2s\va}\,dtdx \nonumber \\
&=
-\frac{C}{2s\la} \int_Q \chi^2 \left[  \int_0^t \left( \sum_{j=1}^n |\pp_j w (\tau, x)|^2 + |w(\tau, x)|^2 \right) \,d\tau \right] \pp_t ( e^{2s\va})\,dtdx \nonumber \\
&=
-\frac{C}{2s\la} \int_Q (\chi(T))^2 \left[  \int_0^T \left( \sum_{j=1}^n |\pp_j w (\tau, x)|^2 + |w(\tau, x)|^2 \right) \,d\tau \right]  e^{2s\va(T)}\,dtdx \nonumber \\
&\quad
+\frac{C}{s\la} \int_Q \chi (\pp_t \chi) \left[  \int_0^t \left( \sum_{j=1}^n |\pp_j w (\tau, x)|^2 + |w(\tau, x)|^2 \right) \,d\tau \right]  e^{2s\va}\,dtdx \nonumber \\
&\quad
+\frac{C}{2s\la} \int_Q \chi^2\left( \sum_{j=1}^n |\pp_j w (t, x)|^2 + |w(t, x)|^2 \right) e^{2s\va}\,dtdx \nonumber \\
&\leq 
C \int_Q \left( |\nabla y |^2 + |y|^2 \right) e^{2s\va}\,dtdx, 
\label{eq:T2}
\end{align}
for $\la>\la_0$ and $s>1$.  \\
Noting that \eqref{eq:rep_v1} -- \eqref{eq:rep_v3},  we get 
\begin{align}
\nonumber
I_3
&=\int_Q \left( \sum_{j=1}^n |\chi \pp_j v|^2 + |\chi v|^2 \right) e^{2s\va}\,dtdx \\
&
=
\int_Q  |\chi w|^2 e^{2s\va}\,dtdx  \nonumber \\
&\quad 
+ \int_Q \left(\sum_{j=2}^n \left|\chi \int_0^{x_1} \pp_j w(x_1,\zeta, x^\prime)\,d\zeta \right|^2 + \left|\chi \int_0^{x_1} w(x_1,\zeta, x^\prime)\,d\zeta\right|^2 \right) e^{2s\va}\,dtdx 
\nonumber \\
&
=
\int_Q  |y|^2 e^{2s\va}\,dtdx \nonumber \\
&\quad 
+ \int_Q \left( \sum_{j=2}^n \left| \int_0^{x_1} \pp_j y(x_1,\zeta, x^\prime)\,d\zeta \right|^2 + \left| \int_0^{x_1} y(x_1,\zeta, x^\prime)\,d\zeta\right|^2 \right) e^{2s\va}\,dtdx 
\nonumber \\
&\leq
\int_Q  |y|^2 e^{2s\va}\,dtdx
\nonumber \\
&\quad 
+ \ell \int_D \int_0^\ell \int_0^T \left( \sum_{j=2}^n \int_0^{\ell} |\pp_j y(x_1,\zeta, x^\prime)|^2 \,d\zeta + \int_0^{\ell} |y(x_1,\zeta, x^\prime)|^2 \,d\zeta \right) e^{2s\va}\,dtdx_1dx^\prime 
\nonumber \\
&=
\int_Q  |y|^2 e^{2s\va}\,dtdx
\nonumber \\
&\quad 
+ \ell^2 \int_D \int_0^\ell \int_0^T \left( \sum_{j=2}^n  |\pp_j y(x_1,\zeta, x^\prime)|^2 +  y(x_1,\zeta, x^\prime)|^2 \right) e^{2s\va}\,dtd\zeta dx^\prime 
\nonumber \\
&\leq 
C\int_Q \left(  |\nabla y|^2 +  |y|^2 \right) e^{2s\va}\,dtdx.  
\label{eq:T3}
\end{align}
By an argument similar to that used to derive \eqref{eq:T1},  
we have 
\begin{align}
\nonumber 
I_4&=\int_Q \left| \chi \int_0^t \sum_{i,j=1}^n |\pp_i \pp_j v (\tau, x)| \,d\tau \right|^2 e^{2s\va}\,dtdx
\\
&\leq 
C
\int_Q  \frac1{s\la \va} \sum_{i,j=1}^n |\chi \pp_i \pp_j v |^2 e^{2s\va}\,dtdx
+
C
\int_Q  \frac1{s^2 \la} \sum_{i,j=1}^n |\chi \pp_i \pp_j v |^2 e^{2s\va}\,dtdx, 
\nonumber 
\end{align}
for $\la>\la_0$ and $s>1$. 
Noting that \eqref{eq:rep_v1}--\eqref{eq:rep_v3}, 
we may estimate the right-hand side of the above inequality 
as similar as \eqref{eq:T3}, that is, 
for $\la>\la_0$ and $s>1$, 
we have
\begin{align*}
&\int_Q  \frac1{s\la \va} \sum_{i,j=1}^n |\chi \pp_i \pp_j v |^2 e^{2s\va}\,dtdx
+
\int_Q  \frac1{s^2 \la} \sum_{i,j=1}^n |\chi \pp_i \pp_j v |^2 e^{2s\va}\,dtdx. 
\\
&\leq 
C
\int_Q  \frac1{s\la \va} \sum_{i,j=1}^n |\pp_i \pp_j y |^2 e^{2s\va}\,dtdx
+
C
\int_Q  \frac1{s^2 \la} \sum_{i,j=1}^n |\pp_i \pp_j y |^2 e^{2s\va}\,dtdx\nonumber \\
&\quad
+
C
\int_Q   |\nabla y |^2 e^{2s\va}\,dtdx.  
\end{align*}
Summing up the above two inequalities, 
for $\la>\la_0$ and $s>1$, 
we obtain
\begin{align}
\nonumber 
I_4&\leq 
C
\int_Q  \frac1{s\la \va} \sum_{i,j=1}^n |\pp_i \pp_j y |^2 e^{2s\va}\,dtdx
+
C
\int_Q  \frac1{s^2 \la} \sum_{i,j=1}^n |\pp_i \pp_j y |^2 e^{2s\va}\,dtdx  \\
&\quad
+
C
\int_Q   |\nabla y |^2 e^{2s\va}\,dtdx. 
\label{eq:T4}
\end{align}
By the same argument as \eqref{eq:T2}, we have
\begin{align}
\nonumber 
I_5&=
\int_Q \left| \chi \int_0^t \left(\sum_{j=1}^n |\pp_j v (\tau, x)| + |v(\tau, x)| \right) \,d\tau\right|^2 e^{2s\va}\,dtdx \\
&\leq 
C\int_Q  \left(\sum_{j=1}^n |\chi \pp_j v|^2 + |\chi v|^2 \right)  e^{2s\va}\,dtdx.  \nonumber 
\end{align}
Combining this 
with \eqref{eq:T3}, we get
\begin{equation}
\label{eq:T5}
I_5\leq 
C\int_Q  \left( |\nabla y|^2 + |y|^2 \right)  e^{2s\va}\,dtdx.   
\end{equation}
Hence, by \eqref{eq:thmce1} and \eqref{eq:T1}--%
\eqref{eq:T5}, 
we have 
\begin{align}
\nonumber 
&\int_Q \left[ \frac{1}{s\va} \left( |\pp_t y|^2 + \sum_{i,j=1}^n |\pp_i\pp_j y|^2 \right) + \la |\nabla y|^2 + s\la^2\va |y|^2 \right] e^{2s\va}\,dtdx\\
&\leq C 
\int_Q |(\pp_t\chi)w|^2 e^{2s\va}\,dtdx 
+C\int_Q \left(  |\nabla y|^2 + \la^2 |y|^2 \right) e^{2s\va}\,dtdx \nonumber \\
&\quad 
+
C
\int_Q  \frac1{s\la \va} \sum_{i,j=1}^n |\pp_i \pp_j y |^2 e^{2s\va}\,dtdx
+
C
\int_Q  \frac1{s^2 \la} \sum_{i,j=1}^n |\pp_i \pp_j y |^2 e^{2s\va}\,dtdx \nonumber \\
&\quad
 +CB_1, \nonumber 
\end{align}
for $\la>\la_0$ and $s>1$. 

Taking sufficiently large $\la>\la_0$, 
and then, 
choosing sufficiently large $s>1$, 
we may absorb the second term to the forth term on the right-hand side of 
the above inequality 
into the left-hand side.  
Hence there exists $\la_1>\la_0$ such that for any $\la>\la_1$, we may take $s_0>1$ satisfying: 
there exists $C^\prime=C^\prime(s_0, \la_1)>$0 such that 
\begin{align}
\nonumber 
&\int_Q \left[ \frac{1}{s\va} \left( |\pp_t y|^2 + \sum_{i,j=1}^n |\pp_i\pp_j y|^2 \right) + \la |\nabla y|^2 + s\la^2\va |y|^2 \right] e^{2s\va}\,dtdx\\
&\leq C^\prime 
\int_Q |(\pp_t\chi)w|^2 e^{2s\va}\,dtdx 
 +C^\prime B_1, 
\label{eq:thmce3}
\end{align}
for all $s>s_0$. 
Let us estimate the first term on the right-hand side of \eqref{eq:thmce3}. 
\begin{align}
\nonumber
\int_Q |(\pp_t\chi)w|^2 e^{2s\va}\,dtdx 
&
\leq 
C^\prime e^{2s\delta_1} \int_\Omega\int_{t_1}^{t_2} |w|^2 \,dtdx   \\
&
\leq 
C^\prime e^{2s\delta_1} \int_\Omega\int_{t_1}^{t_2} \left( |\pp_1 u|^2 + |u|^2 \right) \,dtdx \nonumber \\
&
\leq 
C^\prime M^2 e^{2s\delta_1}. 
\label{eq:thmce3r1}
\end{align}
Next we estimate the boundary terms $B_1$. 
Note that $y=\chi w$, $w=\pp_1 v$, $v=u/R$. 
\begin{align}
 \nonumber 
B_1 
&=
\int_{ \Sigma_1} s\la\va \left(|\pp_t y| + |y| \right) \left|\frac{\pp y}{\pp \nu_{L^\prime}} \right| e^{2s\va}\,dtdS \\
&\leq 
C^\prime e^{C(\la)s} \int_{ \Sigma_1} \left(|\pp_t y| + |y| \right) \left|\frac{\pp y}{\pp \nu_{L^\prime}} \right|\,dtdS \nonumber \\
&\leq 
C^\prime e^{C(\la)s} \int_{ \Sigma_1} \left(|\pp_t w| + |w| \right) \left|\frac{\pp y}{\pp \nu_{L^\prime}} \right|\,dtdS \nonumber \\
&=
C^\prime e^{C(\la)s} \int_{ \Sigma_1} \left(|\pp_t \pp_1 v| + |\pp_1 v| \right) \left|\frac{\pp (\pp_1 v)}{\pp \nu_{L^\prime}}  \right|\,dtdS \nonumber \\
&\leq 
C^\prime e^{C(\la)s} \int_{ \Sigma_1} \left(|\pp_t \pp_1 u| + |\pp_t u| + |\pp_1 u| + |u| \right)  \nonumber \\
&\qquad\qquad\quad
						\times \left( \left|\frac{\pp (\pp_1 u)}{\pp \nu_{L^\prime}}  \right| +\left|\frac{\pp u}{\pp \nu_{L^\prime}} \right| +|\pp_1 u| +|u| \right)\,dtdS \nonumber \\
&=: C^\prime  e^{C(\la)s} B_2. 
\label{eq:thmce3r2}
\end{align}
By \eqref{eq:thmce3}--%
\eqref{eq:thmce3r2}, we obtain
\begin{align}
\nonumber 
&\int_Q \left[ \frac{1}{s\va} \left( |\pp_t y|^2 + \sum_{i,j=1}^n |\pp_i\pp_j y|^2 \right) + \la |\nabla y|^2 + s\la^2\va |y|^2 \right] e^{2s\va}\,dtdx\\
&\leq C^\prime M^2 e^{2s\delta_1}
 +C^\prime B_2 e^{C(\la)s} . 
\label{eq:thmce4}
\end{align}
Henceforth we fix the sufficiently large parameter $\la>\la_1$. 
By \eqref{eq:thmce4}, we see that there exist $s_1>s_0(\la)$ and $\widetilde{C}=\widetilde{C} (s_1, \la)>0$ such that
\begin{align}
\nonumber 
&\int_Q \left[ \frac{1}{s} \left( |\pp_t y|^2 + \sum_{i,j=1}^n |\pp_i\pp_j y|^2 \right) +  |\nabla y|^2 + s |y|^2 \right] e^{2s\va}\,dtdx\\
&\leq \widetilde{C} M^2 e^{2s\delta_1}
 +\widetilde{C}B_2 e^{\widetilde{C}s}, 
\label{eq:thmce5} 
\end{align} 
for all $s>s_1$. 

Taking $L^2((0,t_0)\times \Omega)$ norm to the equation for $v$ \eqref{eq:ibvp_v}, we have

\begin{align}
\nonumber 
&\frac1{s} \| f \|^2_{L^2((0,t_0)\times D)} e^{2s\delta_0}\\
&\leq
\frac{C}{s}
\int_\Omega \int_0^{t_0} 
\Biggl[ \pp_t v- L^\prime v 
-2\sum_{i,j=1}^n \frac1{R} a^\prime_{ij}(\pp_i R) \pp_j v  \nonumber  \\
&\qquad
+\left( \frac{\pp_t R}{R} 
- \sum_{i,j=1}^n \frac1{R} \pp_i (a_{ij}^\prime \pp_j R ) 
- \sum_{j=1}^n \frac1{R} b_j \pp_j R \right) v 
\Biggr]^2 e^{2s\delta_0}\,dtdx \nonumber \\
&\quad
+\frac{C}{s}
\int_\Omega \int_0^{t_0} 
\Biggl\{\int_0^t \frac{R(\tau, x)}{R(t,x)} 
K(t,\tau) v (\tau, x)
\,d\tau
+\int_0^t \frac{1}{R(t,x)} 
\Biggl [2\sum_{i,j=1}^n \wa_{ij}(t,\tau, x) \pp_i R(\tau, x)\pp_j v (\tau, x)\nonumber \\
	&
	\qquad\quad 
	+\sum_{i,j=1}^n  \pp_i (\wa_{ij}(t,\tau, x)\pp_j R (\tau, x)) v(\tau, x)
	+\sum_{j=1}^n \wb_j (t,\tau, x) \pp_j R(\tau,x) v(\tau, x)  
\Biggr]\,d\tau 
\Biggr\}^2 e^{2s\delta_0}\,dtdx\nonumber \\
&\leq
\frac{C}{s}
\int_\Omega \int_0^{t_0} 
\left( |\pp_t v|^2 +\sum_{i,j=1}^n |\pp_i\pp_j v|^2 + \sum_{j=1}^n |\pp_j v|^2 +|v|^2
\right) e^{2s\va}\,dtdx \nonumber \\
&\quad
+\frac{C}{s}
\int_\Omega \int_0^{t_0} 
\int_0^{t_0}  
\left( \sum_{i,j=1}^n |\pp_i\pp_j v(\tau, x)|^2 + \sum_{j=1}^n |\pp_j v(\tau, x)|^2 +|v(\tau, x)|^2 
\right) e^{2s\va(\tau)}
\,d\tau \,dtdx \nonumber \\
&\leq
C
\int_\Omega \int_0^{t_0} 
\left[ \frac1{s} \left(|\pp_t v|^2 +\sum_{i,j=1}^n |\pp_i\pp_j v|^2 \right)+ |\nabla v|^2 +|v|^2
\right] e^{2s\va}\,dtdx \nonumber \\
&\leq
C
\int_\Omega \int_0^{t_0} 
\left[ \frac1{s} \left(|\pp_t y|^2 +\sum_{i,j=1}^n |\pp_i\pp_j y|^2 \right)+  |\nabla y|^2 +|y|^2
\right] e^{2s\va}\,dtdx, \nonumber 
\end{align}
for all $s>s_1$. 
In the last inequality, we used \eqref{eq:rep_v1} -- \eqref{eq:rep_v3} and $w=y$ in $(0,t_0)\times \Omega$. 
Together this 
with \eqref{eq:thmce5}, we obtain 
\begin{equation}
\label{eq:f2}
\frac1{s} \| f \|^2_{L^2((0,t_0)\times D)} e^{2s\delta_0}
\leq \widetilde{C} M^2 e^{2s\delta_1}
 +\widetilde{C}B_2 e^{\widetilde{C}s},
\end{equation}
for all $s>s_1$. 
Dividing the both hand side of \eqref{eq:f2} by $e^{2s\delta_1}/s$, 
for all $s>s_1$, 
\begin{equation*}
 \| f \|^2_{L^2((0,t_0)\times D)}
\leq \widetilde{C}s M^2 e^{-2s(\delta_0-\delta_1)}
 +\widetilde{C}s B_2 e^{\widetilde{C}s}. 
\end{equation*}
Noting that $\delta_0-\delta_1>0$, 
positive constants $C_1, C_2, D_1, D_2>0$ exist such that 
\begin{equation*}
 \| f \|^2_{L^2((0,t_0)\times D)}
\leq s \left( C_1 M^2 e^{-D_1 s}+C_2 B_2 e^{D_2 s}  \right), 
\end{equation*}
for all $s>s_1$.  
Choosing $C_3>0$ such that $C_1 \leq C_3 e^{D_1s_1}$ and $C_2 \leq C_3 e^{-D_2s_1}$ 
and setting $\sigma=s-s_1$, for all $\sigma >0$, we have, 
\begin{equation}
\label{eq:fsigma}
 \| f \|^2_{L^2((0,t_0)\times D)}
\leq C_3(\sigma +s_1) \left( M^2 e^{-D_1 \sigma}+ B_2 e^{D_2 \sigma}  \right). 
\end{equation}

\begin{proof}[Proof of Theorem \ref{thm:uni}]
Since $u(t,x)=0$, $(t,x)\in \Sigma_1$, we have 
\begin{equation*}
\pp_t \pp_1 u(t,x) 
=\pp_t u(t,x)
=\pp_1 u(t,x)
=u(t,x)=0, \quad (t,x)\in \Sigma_1. 
\end{equation*}
Hence, the boundary term vanishes, that is, $B_2=0$. 
Thus we obtain the following estimate by \eqref{eq:fsigma}. 
\begin{equation}
\nonumber 
0\leq 
 \| f \|^2_{L^2((0,t_0)\times D)}
\leq C_3  (\sigma +s_1) M^2 e^{-D_1 \sigma}, 
\end{equation}
for all  $\sigma >0$. As $\sigma$ goes to $\infty$, 
the right-hand side of the above inequality 
tends to $0$. 
Therefore $ \| f \|_{L^2((0,t_0)\times D)}=0$, that is, $f(t,x^\prime)=0$, $(t,x^\prime)\in (0,t_0)\times D$. 
\end{proof}

\begin{proof}[Proof of Theorem \ref{thm:stab}]
Firstly we estimate $B_2$ by $B$. 
By the Cauchy Schwarz inequality, we have
\begin{align}
\nonumber 
B_2
&=
 \int_{ \Sigma_1} \left(|\pp_t \pp_1 u| + |\pp_t u| + |\pp_1 u| + |u| \right) 
 \left( \left|\frac{\pp (\pp_1 u)}{\pp \nu_{L^\prime}}  \right| +\left|\frac{\pp u}{\pp \nu_{L^\prime}} \right| +|\pp_1 u| +|u| \right)\,dtdS \\
&\leq 
C \int_{ \Sigma_1} \left(|\pp_t \pp_1 u|^2 + |\pp_t u|^2 + |\pp_1 u|^2 + |u|^2 + \left|\frac{\pp (\pp_1 u)}{\pp \nu_{L^\prime}}  \right|^2 + \left|\frac{\pp u}{\pp \nu_{L^\prime}} \right|^2\right) \,dtdS \nonumber \\
&\leq C B^2. \nonumber 
\end{align}
Together this with \eqref{eq:fsigma}, we obtain
\begin{equation}
\label{eq:fsigma2}
 \| f \|^2_{L^2((0,t_0)\times D)}
\leq C_3(\sigma +s_1) \left( M^2 e^{-D_1 \sigma}+ B^2 e^{D_2 \sigma}  \right). 
\end{equation}

In the case of $B \geq M$, we may obtain the conditional stability estimate 
by the a priori boundedness $\|f\|_{L^2((0,t_0)\times D)}\leq M$. 

If $B=0$,  by an argument similar to the the proof of the uniqueness (Theorem \ref{thm:uni}), 
we have $f(t,x^\prime)=0$, $(t,x^\prime)\in (0,T)\times D$. 

Thus it is sufficient to suppose that $0<B<M$. 
Choose 
\begin{equation*}
\sigma = -\frac{2 \log \frac{B}{M}}{D_1+D_2}>0.  
\end{equation*}
This $\sigma$ minimizes the right-hand side of \eqref{eq:fsigma2}. 
Setting $\kappa = \frac{D_1}{D_1+D_2}$, we have
\begin{equation*}
 M^2 e^{-D_1 \sigma}+ B^2 e^{D_2 \sigma} 
=2 M^{2(1-\kappa)}B^{2\kappa}. 
\end{equation*}
Combining this with \eqref{eq:fsigma2}, there exists $C_4>0$ such that 
\begin{equation*}
 \| f \|^2_{L^2((0,t_0)\times D)}
\leq C_4 (\sigma +s_1)M^{2(1-\kappa)}B^{2\kappa}. 
\end{equation*}
Thus we conclude \eqref{eq:stab}. 
\end{proof}

%
%
%

\end{document}